\documentclass[11pt]{article}

\usepackage[T1]{fontenc}
\usepackage[latin1]{inputenc}
\usepackage{latexsym}
\usepackage{amssymb}
\usepackage{amsmath}
\usepackage{mathpazo}
\usepackage[dvips]{graphicx}
\usepackage{graphicx}
\usepackage{amsthm}
\usepackage{xcolor}
\usepackage{pgf}
\usepackage{setspace}
\usepackage{float}

\usepackage{tikz}
\usetikzlibrary{arrows,decorations.markings,positioning,automata}
\usetikzlibrary{calc}
\tikzstyle{state}=[circle,thick,draw=black!80]
\def\dR{\mathbb{R}}

\def\calM{\mathcal{M}}
\def\calB{\mathcal{B}}

\usepackage[mathscr]{euscript}
\usepackage[top=3cm,bottom=3cm,right=2.8cm,left=2.8cm,headsep=0.4in]{geometry}

\theoremstyle{plain}
\newtheorem{theorem}{Theorem}

\newtheorem{claim}[theorem]{Claim}

\newtheorem{remark}[theorem]{Remark}

\theoremstyle{definition}
\newtheorem{example}{Example}

\newcommand{\PP}{\mathbb{P}}
\newcommand{\EE}{\mathbb{E}}

\newcommand{\NN}{\mathbb{N}}

\allowdisplaybreaks

\tikzstyle{state}=[circle,thick,draw=black!80,minimum size=30pt]
\tikzstyle{rand}=[circle,thick,draw=black!80,fill=black,inner sep=0pt,minimum size=7pt]

\begin{document}

\title{Reachability and safety objectives in Markov decision processes on long but finite horizons%
\thanks{We are very grateful to Laurent Miclo for a very helpful discussion 
and for pointing us at the related literature in linear algebra. 
We are also grateful to Mickael Randour for a very helpful discussion on reachability and safety objectives and his comments on a draft of this paper, and to Robert Israel and Omri Solan for providing the two proofs of Theorem~\ref{them eigen}.}%
\thanks{This work has been partly supported by COST Action CA16228 European Network for Game Theory.
Ashkenazi-Golan and Solan acknowledge the support of the Israel Science Foundation, grants \#217/17 and \#722/18, 
and the NSFC-ISF Grant \#2510/17.}}
\author{Galit Ashkenazi-Golan%
\footnote{School of Mathematical Sciences, Tel-Aviv University, Tel-Aviv, Israel, 6997800, 
E-mail: galit.ashkenazi@gmail.com.}\and 
J\'{a}nos Flesch\footnote{Department of Quantitative Economics, 
Maastricht University, P.O.Box 616, 6200 MD, The Netherlands. E-mail: j.flesch@maastrichtuniversity.nl.}
\and Arkadi Predtetchinski\footnote{Department of Economics, Maastricht University, P.O.Box 616, 6200 MD, 
The Netherlands. E-mail: a.predtetchinski@maastrichtuniversity.nl.}\and 
Eilon Solan\footnote{School of Mathematical Sciences, Tel-Aviv University, Tel-Aviv, Israel, 6997800, E-mail: eilons@post.tau.ac.il.}}
\maketitle

We consider discrete-time Markov decision processes in which the decision maker is
interested in long but finite horizons. 
First we consider reachability objective: 
the decision maker's goal is to reach a specific target state with the highest possible probability. 
Formally, strategy $\sigma$ overtakes another strategy $\sigma'$, 
if the probability of reaching the target state within horizon $t$ is larger under 
$\sigma$ than under $\sigma'$, for all sufficiently large $t\in\NN$. 
We prove that there exists a pure stationary strategy that is not overtaken by any pure strategy nor 
by any stationary strategy, 
under some condition on the transition structure and respectively under genericity.
A strategy that is not overtaken by any other strategy, called an overtaking optimal strategy, 
does not always exist. We provide sufficient conditions for its existence.

Next we consider safety objective: 
the decision maker's goal is to avoid a specific state with the highest possible probability. 
We argue that the results proven for reachability objective extend to this model.
We finally discuss extensions of our results to two-player zero-sum perfect information games.

\bigskip
\noindent {\emph{JEL  classification:} C73.

\bigskip


\lineskip=1.8pt\baselineskip=18pt\lineskiplimit=0pt \count0=1

}

\section{Introduction}
We consider discrete-time Markov decision processes (MDP) with finite state and action spaces. We consider two different types of objectives for the decision maker: reachability objectives and safety objectives. The decision maker is said to have a reachability objective if his goal is to reach a specific state of the MDP with the highest possible probability, and the decision maker is said to have a safety objective if his goal is the opposite: to avoid a specific state of the MDP with the highest possible probability. Both objectives are standard and have been analyzed extensively in the literature, but they are quite different in nature 
(see, e.g., Baier and Katoen (2008), Chatterjee and Henzinger (2012) and Bruy\`{e}re (2017)).

An important question is on which time horizon the decision maker evaluates his strategies. On any given finite horizon, backward induction guarantees that the decision maker has a pure optimal strategy. Note that this optimal strategy can depend heavily on the horizon, and generally there is no strategy that is optimal on all finite horizons. On the infinite horizon, the decision maker has a pure stationary optimal strategy (cf. Howard (1960) and Blackwell (1962)\footnote{Howard (1960) and Blackwell (1962) consider regular MDPs where the decision maker receives a payoff at every period depending on the state and the action played, and his goal is to maximize the total discounted sum or the long-term average of the payoffs that he receives during the play. Our MDP with a reachability objective can be transformed into a regular MDP (see also Sections 2 and 3) by making the target state absorbing and assigning payoff 1 to the target state and payoff 0 to all other states.}).

In this paper, instead of considering a fixed horizon, we propose to evaluate strategies by how they perform on all long but finite horizons. In particular, such an evaluation can be meaningful if the decision maker knows that the decision process will last long, but he has no information on its exact length. In the case of reachability objectives, such 
an evaluation may also reflect the attitude of a decision maker who is patient and 
can wait for many periods to reach the target state.

More precisely, when the decision maker has a reachability objective with target state $s^*$, we say that a strategy $\sigma$ overtakes another strategy $\sigma'$ if there exists $T\in\NN$ such that, on all finite horizons $t\geq T$, the probability of having visited the state $s^*$ within horizon $t$ is strictly larger under $\sigma$ than under $\sigma'$. Thus, conditionally on the MDP lasting at least $T$ periods, the strategy $\sigma$ performs better than $\sigma'$ 
regardless of the horizon, and consequently the decision maker should prefer $\sigma$ to $\sigma'$. When the decision maker has a safety objective and wants to avoid a state $s^*$, we say that a strategy $\sigma$ overtakes another strategy $\sigma'$ if there exists $T\in\NN$ such that, on all finite horizons $t\geq T$, the probability of having visited the state $s^*$ within horizon $t$ is strictly smaller under $\sigma$ than under $\sigma'$.

We also define a more permissive version of the aforementioned relations between strategies. For reachability objectives, we say that a strategy $\sigma$ weakly overtakes another strategy $\sigma'$ if there exists $T\in\NN$ such that, on all finite horizons $t\geq T$, the probability of having visited the target state $s^*$ within horizon $t$ under $\sigma$ is at least as much as that under $\sigma'$, but strictly more for infinitely many horizons $t$. A similar definition can be given for safety objectives.

Under these comparisons of strategies, we call a strategy overtaking optimal if it is not overtaken by any other strategy, and call it strongly overtaking optimal if it is not weakly overtaken by any other strategy. Strong overtaking optimality is a strict refinement of overtaking optimality, and as an appealing property, they are both strict refinements of optimality on the infinite horizon. \medskip

\noindent\textbf{Our contribution.} Our main results can be summarized as follows.

(I) For reachability objectives, we obtain the following results, sorted by the attributes of the MDP:

(I.1) We prove that if the MDP is such that each action can lead to at most one non-target state with a positive probability, then there exists a pure stationary strategy that is not weakly overtaken by any pure strategy. This is Theorem~\ref{theorem-det}. We show with Example \ref{example each pure} that such a statement does not hold for all MDPs.

(I.2) What does hold for all MDPs is Theorem~\ref{theorem non deter}, stating that if, from every initial state, there is a strategy that weakly overtakes all other strategies, then there exists a stationary strongly overtaking optimal strategy.

(I.3) We show by means of Example \ref{example: 2} that an overtaking optimal strategy does not always exist. This MDP  is however constructed in a very specific way and has some typically non-generic properties.

(I.4) We consider MDPs that are ``generic'', in the sense that the transition probabilities are randomized using any non-trivial joint density function. We show for these MDPs that there exists a pure stationary strategy that overtakes each other stationary strategy. This is Theorem \ref{gen-thm}. 

(II) For safety objectives, we argue that the same results hold.

Besides, we briefly discuss extensions to two-player zero-sum perfect information games. In these games, each state is controlled by one of the players. Player 1's goal is to reach a specific state (reachability objective), and player 2's goal is to prevent it (safety objective).
\medskip

\noindent\textbf{Proof techniques.} We use quite different proof techniques to obtain our results. For proving result (I.1), we transform the MDP with the reachability objective into a regular MDP, by assigning payoffs to actions based on the immediate transition probabilities to the target state. In this new MDP, we invoke some results in Flesch et al.\ (2017) to derive a specific pure stationary strategy. We show that this strategy is exactly the desired strategy in the original MDP with the reachability objective. This proof technique is suitable for pure strategies, but probably also limited to them, as the relation between the two MDPs is much weaker for non-pure strategies.

When considering generic MDPs in result (I.4), we rely on techniques from linear algebra. The overtaking comparison between two stationary strategies can be reduced to the comparison of the spectral gaps of the transition matrices that these strategies induce. The spectral gap of a transition matrix refers 
to the difference between the largest eigenvalue, which is equal to 1, and the modulus of the second eigenvalue, which can be a complex number. 
To obtain result (I.4), we need to compare the spectral gaps of transition matrices induced by stationary and pure stationary strategies.

Result (I.2) is proven in a constructive way. The mixed actions of a strategy that weakly overtakes all other strategies can be used for the mixed actions of the desired stationary strategy.

The results for (II) are proven similarly.\bigskip

\noindent \textbf{Related literature.} Reachability and safety problems were studied both in the MDP framework and in the context of two-player zero-sum games, for an overview we refer to Baier and Katoen (2008) and respectively to Chatterjee and Henzinger (2012) and Bruy\`{e}re (2017). An important distinction is made between the qualitative and the quantitative approaches. The qualitative approach is interested in the probability with which the decision maker 
succeeds to meet his objective. For the quantitative approach, however, it also matters how quickly the target state is reached in the case of a reachability objective, or how long the bad state has been avoided in the case of a safety objective. Our overtaking approach could thus be classified as a quantitative approach on the infinite horizon. For other quantitative approaches, we refer to Randour et al. (2015, 2017) and the references therein, and to Brihaye et al. (2014) in a game setting.

In the game-theoretic literature, various definitions of overtaking optimality have been proposed. They all serve as a refinement of optimality on the infinite horizon, based on the performance of strategies on the finite horizons. For an overview of some of these concepts, we refer to Stern (1984), Puterman (1994), Carlson et al (1991), Zaslavski (2006, 2014), Guo and Hern\`{a}ndez--Lerma (2009), and M\'{e}der et al. (2012). Our definition of overtaking optimality is a relatively direct translation of the definitions of sporadic overtaking optimality in Stern (1984) and Flesch et al. (2017) and repeated optimality in M\'{e}der et al. (2012), into the context of MDPs with reachability and safety objectives.

One important feature of our definition of overtaking optimality is that it does not require the strategy to outperform all other strategies on long but finite horizons. It only requires that the strategy is not outperformed by any other strategy. Our definition of overtaking optimality is therefore weaker than overtaking optimality and uniform overtaking optimality as in Stern (1984), and weaker than strong overtaking optimality as in Nowak and Vega-Amaya (1999) 
or Leizarowitz (1996). See also M\'{e}der et al. (2012), who delineates ''not-outperformed'' definitions of optimality from ``outperform-all'' definitions of optimality. \bigskip

\noindent \textbf{Organization of the paper.} Section 2 details the model. 
Then, we start by analyzing reachability objectives. Section~3 provides an example which highlights different aspects of the concept of overtaking optimality by comparing it with other optimality notions. Section~4 presents the results concerning deterministic MDPs. Section~5 considers non-deterministic MDPs. Section~6 treats generic MDPs. 
In Section~7 we turn our attention to safety objectives. Section~8 concludes.

\section{The model}

\subsection{MDPs with a reachability objective}

\noindent\textbf{The model.} An \emph{MDP with a reachability objective} is given by [1] a finite set $S$ of states, with a specific state $s^*\in S$, called the target state, [2] for each state $s\in S$, a nonempty and finite set $A(s)$ of actions, and [3] for each state $s\in S$ and each action $a\in A(s)$, a probability distribution $p(s,a)$ on the set $S$ of states.

An MDP with a reachability objective is played at periods in $\mathbb{N}=\{1,2,\ldots\}$ as follows: At period 1, in a given initial state $s_1\in S\setminus\{s^*\}$, the decision maker chooses an action $a_1\in A(s_1)$, which leads to a state $s_2 \in S$ drawn from the distribution $p(s_1,a_1)$. At period 2, in state $s_2$, the decision maker chooses an action $a_2\in A(s_2)$, leading to state $s_3$ drawn from $p(s_2,a_2)$, and so on. The decision maker's goal is that state $s^*$ is eventually reached, that is, $s_t=s^*$ for some $t\in\NN$.\medskip

\noindent\textbf{Histories.} The \emph{history at period $t\in\mathbb{N}$} is a sequence $(s_1,a_1,\ldots,s_{t-1},a_{t-1},s_t)$ such that $s_i \in S$ for all $i=1,\ldots,t$, and $a_i\in A(s_i)$ and $p(s_{i+1}|s_i,a_i) > 0$ for all $i=1,\ldots,t-1$. We denote by $H_t$ the set of all histories at period $t$, and by $H=\cup_{t\in \mathbb{N}}H_t$ the set of all histories. Let $s(h)$ denote the final state of each history $h\in H$.
Let $H_{\infty}$ be the set of all \emph{infinite histories}, i.e., the set of sequences $(s_1,a_1,s_2,a_2,\dots)$ such that $s_i \in S$, $a_i \in A(s_i)$, and $p(s_{i+1}|s_i,a_i) > 0$ for each $i \in \mathbb{N}$.\medskip

\noindent\textbf{Strategies.} A \emph{mixed action} in a state $s\in S$ is a probability distribution on $A(s)$. The set of mixed actions in state $s$ is denoted by $\Delta(A(s))$.

A \emph{strategy} $\sigma$ is a map that, to each history $h\in H$, assigns a mixed action $\sigma(h)\in \Delta(A(s(h)))$. The interpretation is that, if history $h$ arises, $\sigma$ recommends to choose an action in the current state $s(h)$ according to the probabilities given by the mixed action $\sigma(h)$.
A strategy $\sigma$ is called \emph{pure} if $\sigma(h)$ places probability 1 on one action, for each history $h$. A strategy $\sigma$ is called \emph{stationary} if the recommendation of the action only depends on the current state, i.e., $\sigma(h)=\sigma(h')$ whenever $s(h)=s(h')$. Note that a pure stationary strategy can be seen as an element of $\times_{s\in S}A(s)$.

Starting from an initial state $s$, a strategy $\sigma$ induces a  probability measure $\mathbb{P}_{s\sigma}$ on $H_{\infty}$, where $H_{\infty}$ is endowed with the sigma--algebra generated by the cylinder sets. We denote the corresponding expectation operator by $\mathbb{E}_{s\sigma}$.\medskip

\noindent\textbf{Value and optimality.} Let $t^*$ denote the first period when state $s^*$ is reached. If $s^*$ is never reached then we define $t^*=\infty$. For initial state $s\in S$, the \emph{value} $v(s)$ is defined as $v(s)=\sup_{\sigma}\PP_{s\sigma}(t^*<\infty)$, which is the maximal probability that state $s^*$ can be reached. A strategy $\sigma$ is called 
\emph{optimal} at the initial state $s$ if $\PP_{s\sigma}(t^*<\infty)=v(s)$. 
Note that, for any strategy $\sigma$, 
it holds that as $t$ goes to infinity, 
the quantity $\PP_{s\sigma}(t^*\leq t)$ increases and converges to $\PP_{s\sigma}(t^*<\infty)$.
Thus, $\sigma$ is optimal at the initial state $s$ if and only if $\lim_{t\to\infty}\PP_{s\sigma}(t^*\leq t)=v(s)$. 
It is known that the decision maker always has a pure stationary strategy 
that is optimal at all initial states (cf. Howard (1960) and Blackwell (1962)). \medskip

\noindent\textbf{Overtaking optimality.} 
We say that a strategy $\sigma$ \emph{overtakes} a strategy $\sigma'$ at the initial state $s$ 
if there is $T\in\NN$ such that for all periods $t\geq T$ we have $\PP_{s\sigma}(t^*\leq t)\,>\, \PP_{s\sigma'}(t^*\leq t)$. This means that, for all periods $t\geq T$, the probability under $\sigma$ to reach $s^*$ within the first $t$ periods is strictly larger than that under $\sigma'$. If the decision maker is sufficiently patient with regard to his goal to reach the target state, then he strictly prefers $\sigma$ to $\sigma'$.

Note that two strategies can be incomparable in the sense that neither of them overtakes the other one. Indeed, consider the following example. The state space is $\{x,s^*\}$. In state $x$, the decision maker has three actions: $a_0,a_{1/2},a_{7/8}$. 
For $z \in \{0,1/2, 7/8\}$, under action $a_{z}$, 
the play moves to state $s^*$ with probability $z$, 
and thus remains in state $x$ with probability $1-z$. Now suppose that $\sigma$ recommends to always play action $a_{1/2}$, and $\sigma'$ recommends to play the sequence of actions $a_0,a_{7/8},a_0,a_0,a_{7/8},a_0,\ldots$ as long as the play is in state $x$. Then, at periods $t=3k+1$, where $k\in\NN$, we have $\PP_{s\sigma}(t^*\leq t)\,=\, \PP_{s\sigma'}(t^*\leq t)\,=\,(7/8)^k$. At periods $t=3k+2$, we have $\PP_{s\sigma}(t^*\leq t)\,>\, \PP_{s\sigma'}(t^*\leq t)$. At periods $t=3k$, we have $\PP_{s\sigma}(t^*\leq t)\,<\, \PP_{s\sigma'}(t^*\leq t)$. So, $\sigma$ and $\sigma'$ are incomparable.

A strategy $\sigma$ is called \emph{overtaking optimal} at the initial state $s$ if there is no strategy that overtakes $\sigma$ at that initial state. That is, $\sigma$ is maximal with respect to the relation of ``overtakes'' between strategies.

Note that any optimal strategy overtakes any strategy that is not optimal. 
Indeed, if $\sigma$ is optimal at the initial state $s$ 
but $\sigma'$ is not, then $\PP_{s\sigma}(t^*<\infty)\,=v(s)\,>\, \PP_{s\sigma'}(t^*<\infty)$, 
and hence $\PP_{s\sigma}(t^*\leq t)\,>\, \PP_{s\sigma'}(t^*\leq t)$ for all sufficiently large $t$. 
Consequently, an overtaking optimal strategy at the initial state $s$ is also optimal at that initial state. 
As Example~\ref{example: 1} below will show, the converse is not true: 
there exist optimal strategies at an initial state
that are not overtaking optimal at that initial state. 
Thus, overtaking optimality is a strict refinement of optimality.
\medskip

\noindent\textbf{Strong overtaking optimality.}
We say that a strategy $\sigma$ \emph{weakly overtakes} a strategy $\sigma'$ at the initial state $s$
if there is $T\in\NN$ such that for all periods $t\geq T$ we have 
$\PP_{s\sigma}(t^*\leq t)\,\geq\, \PP_{s\sigma'}(t^*\leq t)$ 
with strict inequality for infinitely many $t$. 
Note that if $\sigma$ overtakes $\sigma'$ at the initial state $s$ 
then $\sigma$ also weakly overtakes $\sigma'$ at that initial state.

A strategy $\sigma$ is called \emph{strongly overtaking optimal} at the initial state $s$ 
if there is no strategy that weakly overtakes $\sigma$ at that initial state. 
A strongly overtaking optimal strategy at an initial state is also overtaking optimal at that initial state. \medskip

\subsection{MDPs with a safety objective}

The model of MDPs with a safety objective is very similar to the model of MDPs with a reachability objective, 
except that the decision maker's objective is to reach the state $s^*$ with as low a
probability as possible.
\medskip

\noindent\textbf{Value and optimality.} 
For initial state $s\in S$, the \emph{value} $v(s)$ is defined as 
$v(s)=\inf_{\sigma}\PP_{s\sigma}(t^*<\infty)$. 
A strategy $\sigma$ is called \emph{optimal} at the initial state $s$ 
if $\PP_{s\sigma}(t^*<\infty)=v(s)$. 
Also in this model it is known that the decision maker always has a pure stationary strategy 
that is optimal at all initial states (cf. Howard (1960) and Blackwell (1962)). 
\medskip

\noindent\textbf{Overtaking optimality.} 
We say that a strategy $\sigma$ \emph{overtakes} a strategy $\sigma'$ at the initial state $s$
if there is $T\in\NN$ such that for all periods $t\geq T$ we have 
$\PP_{s\sigma}(t^*\leq t)\,<\,\PP_{s\sigma'}(t^*\leq t)$. 
A strategy $\sigma$ is called \emph{overtaking optimal} at the initial state $s$ 
if there is no strategy that overtakes $\sigma$ at that initial state. \medskip

\noindent\textbf{Strong overtaking optimality.} 
We say that a strategy $\sigma$ \emph{weakly overtakes} a strategy $\sigma'$ at the initial state $s$ 
if there is $T\in\NN$ such that for all periods $t\geq T$ we have 
$\PP_{s\sigma}(t^*\leq t)\,\leq\, \PP_{s\sigma'}(t^*\leq t)$ with strict inequality for infinitely many $t$. 
Note that if $\sigma$ overtakes $\sigma'$ at the initial state $s$ 
then $\sigma$ also weakly overtakes $\sigma'$ at that initial state.

A strategy $\sigma$ is called \emph{strongly overtaking optimal} at the initial state $s$ 
if there is no strategy that weakly overtakes $\sigma$ at that initial state. 
A strongly overtaking optimal strategy at the initial state $s$ is also overtaking optimal 
at that initial state.

\subsection{Discounted and average payoff MDPs} 
We will also consider MDPs with the discounted payoff or with the average payoff, but only as auxiliary models in our analysis.

A \emph{discounted MDP} can be described similarly to an MDP with a reachability or safety objective, 
but with the following modifications: (1) For each state $s\in S$ and each action $a\in A(s)$, 
there is a payoff $u(s,a)\in\mathbb{R}$, which the decision maker receives when action $a$ 
is played in state $s$. 
(2) The decision maker's goal is not to reach a target state $s^*$, 
but rather to maximize the expectation of the total discounted payoff
\[(1-\beta)\cdot(u(s_1,a_1)+\beta\cdot u(s_2,a_2)+\beta^2\cdot u(s_3,a_3)+\cdots), \]
where $\beta\in(0,1)$ is the discount factor. 
Strategies, the discounted-value, and discounted-optimality are defined 
analogously to the corresponding definitions for reachability objective. 
By the results of Shapley (1953) and Blackwell (1962), 
it is known that for each discount factor $\beta\in(0,1)$, 
the decision maker has a pure stationary strategy that is $\beta$-discounted optimal at all initial states. 
Moreover, the decision maker has a pure stationary strategy that is 
$\beta$-discounted optimal for all $\beta$ close to 1 and all initial states. 
Such a strategy is also called \emph{Blackwell optimal}.

An \emph{average payoff MDP} is similar to a discounted MDP, except that (2') the decision maker's goal is to maximize the expectation of the average payoff
\[\liminf_{T\to\infty}\ \frac{1}{T}\sum_{t=1}^T u(s_t,a_t).\]
The average-value, and average-optimality are defined
analogously to the corresponding definitions for reachability objective. 
By Howard (1960) and Blackwell (1962), 
it is known that the decision maker has a pure stationary strategy
that is average-optimal at all initial states. 
Moreover, each Blackwell optimal strategy is average optimal at all initial states.

\section{Reachability objectives: An illustrative example}
In this section, we discuss a specific MDP with a reachability objective, which demonstrates three properties of overtaking optimality.
First, there are optimal strategies that are not overtaking optimal. That is, overtaking optimality is a strict refinement of optimality.
Second, finding overtaking optimal strategies cannot be done by simply solving a related discounted MDP.
Third, the strategy that minimizes the expected time of reaching the target state $s^*$ can be different from the overtaking optimal strategies, even when the latter is unique.
 \medskip

\begin{example}\label{example: 1}
Consider an MDP with a reachability objective which has state space $S=\{x,y,z,s^*\}$ such that:
\begin{itemize}
    \item State $x$ is the initial state. In this state, the decision maker has two actions: $a$ and $b$. Action $a$ leads to state $s^*$ with probability $q$ and to state $y$ with probability $1-q$. Action $b$ leads to state $s^*$ with probability $\frac{1}{2}$ and to state $z$ with probability $\frac{1}{2}$.
    \item In state $y$, there is only one action, denoted by $c$, which leads to state $s^*$ with probability $q$ and to state $y$ with probability $1-q$.
    \item In state $z$, there is only one action, denoted by $d$, which leads to state $s^*$ with probability $p$ and to state $z$ with probability $1-p$.
    \item State $s^*$ is the state that the decision maker tries to reach. This state is absorbing.
\end{itemize}
The probabilities $p$ and $q$ are such that $0<p<q<\frac{2p}{2p+1}$. For example, $p=0.1$ and $q=0.11$. Note that $p<\frac{2p}{2p+1}$ implies $p<1/2$, and so $q<\frac{2p}{2p+1}$ implies $q<1/2$ as well.

The MDP is depicted in Figure \ref{fig2}. 
In this figure, the state $s^*$ is omitted for simplicity, 
and the names of the actions are provided together with the corresponding probabilities of 
moving to state $s^*$.

\begin{figure}[H]\label{fig2}
\begin{center}
\begin{tikzpicture}[scale=0.8,>=stealth',shorten >=1pt,auto,node distance=2.8cm]
\node at(0,3) [state] (s1) {$x$};
\node at(6,6) [state] (s2) {$y$};
\node at(6,0) [state] (s3) {$z$};
\path[->,thick]
 (s2) edge[loop] node[above] {$c:q$} ()
 (s3) edge[loop] node[above] {$d:p$} ()
 (s1) edge node[above] {$a:q$} (s2)
 (s1) edge node[above] {$b:\frac{1}{2}$}(s3);
\end{tikzpicture}
\end{center}
\caption{The MDP in Example~\ref{fig2}.}
\end{figure}
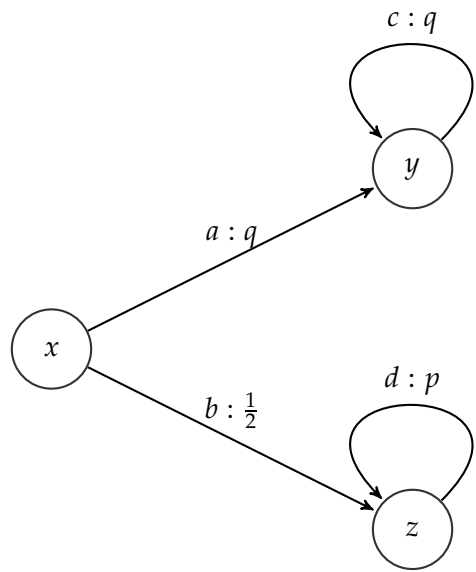
\bigskip

\noindent

Suppose that the initial state is $x$.
Choosing the action $a$ at state $x$ leads to the following sequence of probabilities 
of moving to state $s^*$: $(q,q,q,\ldots)$.
Choosing the action $b$ at state $x$ leads to the following sequence of probabilities 
of moving to state $s^*$: $(1/2,p,p,\ldots)$. 
The target state $s^*$ is eventually reached with probability 1 under both actions $a$ and $b$. 
Thus, both actions are optimal in the reachability problem.

\begin{claim}
In Example \ref{example: 1}, action $a$ is overtaking optimal, and action $b$ is not.
\end{claim}

\begin{proof}
Take a period $t\geq 2$. Then, under $a$ the probability of reaching the target state $s^*$ within the first $t$ periods is
\[\PP_{xa}(t^*\leq t)\;=\;1-(1-q)^{t-1},\]
whereas under $b$ this probability is
\[\PP_{xb}(t^*\leq t)\;=\;\frac{1}{2}+\frac{1}{2}\cdot(1-(1-p)^{t-2}).\]
Thus,
\begin{align*}
    \PP_{xa}(t^*\leq t)-\PP_{xb}(t^*\leq t) & \;=\;-(1-q)^{t-1}+\frac{1}{2}\cdot(1-p)^{t-2}\\
    &\;=\;(1-p)^{t-2}\cdot\left[-\left(\frac{1-q}{1-p}\right)^{t-2}\cdot(1-q)+\frac{1}{2}\right].
\end{align*}
Since $p<q$ by assumption, we have $(1-q)/(1-p)<1$. 
Hence, $\PP_{xa}(t^*\leq t)-\PP_{xb}(t^*\leq t)$ is positive for large $t\in\NN$. 
This completes the proof.
\end{proof}

\begin{claim}
Consider a discounted MDP which has the same state space, action spaces and transitions as the MDP of Example \ref{example: 1}, and has payoff equal to 1 in state $s^*$ and payoff 0 in all other states.

Then, for all discount factors $\beta$ close to 1, action $b$ leads to a strictly higher expected $\beta$-discounted payoff than action $a$.
\end{claim}

\begin{proof}
Note that in the discounted MDP, if state $s^*$ is reached at period $t$, then the $\beta$-discounted payoff is equal to $(1-\beta)\cdot(\beta^{t-1}+\beta^t+\cdots)=\beta^{t-1}.$

Thus, action $a$ leads to the expected discounted payoff:
\[D_a(\beta)\;=\;q\beta+(1-q)q\beta^2+(1-q)^2 q\beta^3+\cdots\;=\;q\beta\frac{1}{1-(1-q)\beta}\;,\]
whereas action $b$ leads to the following expected discounted payoff:
\[D_b(\beta)\;=\;\frac{1}{2}\beta+\frac{1}{2}\cdot(p\beta^2+(1-p)p\beta^3+(1-p)^2 p\beta^4+\cdots)\;=\;\frac{1}{2}\beta+\frac{1}{2}p\beta^2\frac{1}{1-(1-p)\beta}.\]
As one can verify, we have
\[D_b(\beta)-D_a(\beta)\;=\;\frac{\frac{1}{2}\beta(1-\beta)\cdot[1-(1-2p)(1-q)\beta-2q]}{(1-(1-p)\beta)\cdot(1-(1-q)\beta)}.\]
The denominator of the fraction above is positive for all $\beta\in(0,1)$. Since $q<\frac{2p}{2p+1}$ by assumption, the expression $1-(1-2p)(1-q)-2q$ is positive. Hence, the numerator of the fraction above is positive for large $\beta\in(0,1)$. Thus, the fraction above is also positive for large $\beta\in(0,1)$, which implies the claim that $D_b(\beta)>D_a(\beta)$ for large $\beta\in(0,1)$.
\end{proof}

\begin{claim}
In Example \ref{example: 1}, the expectation of the period $t^*$ when reaching the state $s^*$ 
is smaller under $b$ than under $a$: \[\EE_{xb}(t^*)\,<\,\EE_{xa}(t^*).\]
\end{claim}

\begin{proof}
 We have
\[\EE_{xa}(t^*)\;=\;q+2(1-q)q+3q(1-q)^2+\cdots\;=\;\frac{1}{q}\]
and
\[\EE_{xb}(t^*)\;=\;\frac{1}{2}+\frac{1}{2}\cdot[2p+3(1-p)p+4p(1-p)^2+\cdots]\;=\;1+\frac{1}{2p}.\]
Since $q<\frac{2p}{2p+1}$ by assumption, the claim follows.
\end{proof}

\end{example}



\section{Reachability objectives: Deterministic MDPs}

We call an MDP with a reachability objective \textit{deterministic} 
if for each state $s\neq s^*$ and every action $a\in A(s)$
there is a state $z \in S$ such that $p(\{z,s^*\}|s,a)=1$. 
That is, for any combination of state and action, 
the play moves to the target state $s^*$
or to a specific state that is not $s^*$.

The following theorem claims that, in deterministic MDPs with a reachability objective, there always exists a pure stationary strategy that is at least as good as any pure strategy in the overtaking sense. The main idea of the proof is to transform the MDP with a reachability objective into an average payoff MDP. The payoffs that we assign to actions are related to the probabilities that these actions lead to the target state.

The condition that the MDP is deterministic plays an important role. 
Indeed, Example \ref{example each pure} will show that the result is not true in general for non-deterministic MPDs.

\begin{theorem}\label{theorem-det}
In every deterministic MDP with a reachability objective, there exists a pure stationary strategy that is not weakly overtaken by any other pure strategy.
\end{theorem}

For this purpose, we need the following theorem:

\begin{theorem}\label{extended spor}
Consider an average-payoff MDP such that\footnote{This is the standard definition of a deterministic MDP.} 
in each state $s\in S$ each action $a\in A(s)$ leads to a certain state $\omega(s,a) \in S$ with probability 1: 
$p(\omega(s,a)|s,a)=1$. Let $\sigma$ be a Blackwell optimal strategy. 
Then, there exists no pure strategy $\sigma'$ and no initial state $s \in S$
with the following properties:
\begin{itemize}
	\item there is $M\in\NN$ such that for all periods $t\geq M$ we have $u_t(s,\sigma)\,\leq\,u_t(s,\sigma')$, where $u_t(s,\sigma)$ and $u_t(s,\sigma')$ denote the expected average payoffs up to period $t$ under $\sigma$ and respectively under $\sigma'$ at initial state $s$,
	\item and $u_t(s,\sigma)\,<\,u_t(s,\sigma')$ holds for infinitely many $t$.
\end{itemize}
\end{theorem}

\begin{proof} (of Theorem~\ref{extended spor})
 We closely follow the proof of Theorem 1 (that [1] implies [4], page 219) in Flesch et al (2017).

Fix the initial state $s$. Without loss of generality assume that $v(s) = 0$.

\noindent A sequence $\ell=(s_1,a_1,s_2,\cdots,s_t,a_t,s_{t+1})$ is called a \emph{loop} if (1) $s_1,\ldots,s_t$ are distinct elements of $S$, (2) $s_1=s_{t+1}$, (3) $a_i\in A(s_i)$ for all $i=1,\ldots,t$, and (4) $p(s_{i+1}|s_i ,a_i)=1$ for all $i=1,\ldots,t$. For a loop $\ell=(s_1,a_1,s_2,\cdots,s_t,a_t,s_{t+1})$ let $\phi(\ell)$ denote the sum of the payoffs along $\ell$:
\[ \phi(\ell) = \sum_{n=1}^t u(s_t,a_t). \]
The number of loops is finite and therefore the following quantity $\delta$ is negative:
\begin{equation}\label{delta}
\delta = \max\{\phi(\ell):\ell\text{ is a loop and }\phi(\ell) < 0\}
\end{equation}
(the definition of $\delta$ is irrelevant if the set over which the maximum is taken is empty).

Suppose that the strategy $\sigma$ is Blackwell optimal for the initial state $s$. Let $s_{1},a_{1},s_{2},a_{2},\dots$ be the sequence of states and actions induced by the strategy $\sigma$, where $s_{1} = s$.
Denote $u_{t} = u(s_{t},a_{t})$.

Take any pure strategy $\sigma'$. We let $s_{1}',a_{1}',s_{2}',a_{2}',\dots$ be the sequence of states and actions induced by the strategy $\sigma'$, where $s_{1}' = s$,
and denote $u_{t}' = u(s_{t}',a_{t}')$.

The following two claims have been proven in Flesch et al (2017, see page 220).\medskip

\noindent\textbf{Claim 1:} Suppose that for some $k,m \in \mathbb{N}$ we have $s_{k} = s_{k+m+1}$. 
Then $u_{k} + \cdots + u_{k+m} = 0$.\medskip

\noindent\textbf{Claim 2:} Suppose that for some $k,m \in \mathbb{N}$ we have $s_{k}'=s_{k+m+1}'$. 
Then either $u_{k}' + \dots + u_{k + m}' = 0$ or $u_{k}' + \dots + u_{k + m}' \leq \delta$ 
(recall that $\delta < 0$).\medskip

Suppose by way of contradiction that there is $M \in \mathbb{N}$ such that $u_{T}(s,\sigma') \geq u_{T}(s,\sigma)$ for all $T \geq M$ and the inequality is strict for an infinite sequence $T_1<T_2<\ldots$, where $M\leq T_1$.

Let $q_{1} = 0$ and $q_{t} = u_{1} + \cdots + u_{t - 1}$ for each $t \geq 1$.
Likewise, let $q_{1}' = 0$ and $q_{t}' = u_{1}' + \cdots + u_{t - 1}'$ for each $t \geq 1$. Due to our assumption about $\sigma'$, we have $q_{t}' - q_{t} > 0$  for each $t=T_1,T_2,\ldots$.

We next argue that there exists $\mu > 0$ and such that
\begin{equation}\label{bound}
q_{t}' - q_{t} \geq \mu\text{ for each }t=T_1,T_2,\ldots.
\end{equation}
Indeed, suppose that no such $\mu > 0$ exists. 
Then, by taking a subsequence if necessary, we can assume that the sequence $\{q_{T_{n}}' - q_{T_{n}}\}_{n \in \mathbb{N}}$ is strictly decreasing. That is, $q_{T_{n}}' - q_{T_{n}} > q_{T_{n + 1}}' - q_{T_{n + 1}}$ for every $n \in \mathbb{N}$. As is shown in Flesch et al (2017, page 220), by claims 1 and 2 above, this leads to $q_{T_{n}}' - q_{T_{n}} < 0$ for all sufficiently large $n$, which is a contradiction.

For each discount factor $\beta$, let $u_\beta(s,\sigma)$ and $u_\beta(s,\sigma')$ denote the expected discounted payoffs under $\sigma$ and respectively under $\sigma'$ at initial state $s$. Then, 
\begin{align*}
\frac{u_{\beta}(s,\sigma)}{1 - \beta} &= \sum_{t = 1}^{\infty}\beta^{t - 1}u_{t}\\
&= \sum_{t = 1}^{\infty}\beta^{t - 1}(q_{t + 1} - q_{t})\\
&= \sum_{t = 2}^{\infty}\beta^{t - 2}q_{t} - \sum_{t = 1}^{\infty}\beta^{t - 1}q_{t}\\
&= \sum_{t = 2}^{\infty}(\beta^{t - 2}-\beta^{t - 1})q_{t}.
\end{align*}
Likewise,
\begin{align*}
\frac{u_{\beta}(s,\sigma')}{1 - \beta} &= \sum_{t = 2}^{\infty}(\beta^{t - 2}-\beta^{t - 1})q_{t}'.
\end{align*}
Take any $\beta \in (0,1)$ for which $\sigma$ is $\beta$-optimal for the initial state $s$. 
Then, by Eq.~\eqref{bound}, we obtain
\begin{align*}
0 &\geq \frac{u_{\beta}(s,\sigma') - u_{\beta}(s,\sigma)}{1-\beta}\\
&=\sum_{t = 2}^{\infty}(\beta^{t - 2}-\beta^{t - 1})(q_{t}' - q_{t})\\
&\geq \sum_{t = 2}^{T_1}(\beta^{t - 2}-\beta^{t - 1})(q_{t}' - q_{t}) + \sum_{k=1}^\infty\sum_{t =T_k+1}^{T_{k+1}}(\beta^{t - 2}-\beta^{t - 1})\mu\\
&=\sum_{t = 2}^{T_1}(\beta^{t - 2}-\beta^{t - 1})(q_{t}' - q_{t}) + \beta^{T_1-1}\mu.
\end{align*}
Finally, taking the limit of the last inequality as $\beta \uparrow 1$, we obtain $0 \geq \mu$, a contradiction.
\end{proof}

\begin{proof} (of Theorem~\ref{theorem-det}) 
Consider a deterministic MDP with a reachability objective, denoted by $\mathcal{M}$. We may assume without loss of generality\footnote{Indeed, in such a state $s$ it is optimal to choose such an action $a$, as it leads in one step to state $s^*$. Hence, such a state can be deleted from the MDP, and each transition to state $s$ can be rewritten as a transition directly to $s^*$.} that there is no state $s\neq s^*$ and action $a\in A(s)$ such that $p(s^*|s,a)=1$. As the MDP $\mathcal{M}$ is deterministic, this implies that, if $z\in S$ denotes the unique state for which $p(\{z,s^*\}|s,a)=1$, then $p(z|s,a)>0$.

We define an auxiliary average payoff MDP $\mathcal{M}'$ as follows:
\begin{itemize}
    \item The state space is $S'=S-\{s^*\}$.
    \item For each state $s\in S'$, the action space is the same as in the MDP $\mathcal{M}$ with reachability objective: 
    $A'(s)=A(s)$.
    \item For each state $s\in S'$ and action $a\in A(s)$, the transition and the payoff in $\mathcal{M}'$ are defined as follows: If $z\neq s^*$ denotes the unique state for which $p(\{z,s^*\}|s,a)=1$, then $p'(z|s,a)=1$ and $u'(s,a)=-\log(p(z|s,a))$.
\end{itemize}

Intuitively, the MDP $\mathcal{M}'$ represents what happens if in the original MDP $\mathcal{M}$ the decision maker is unlucky at all periods: if moving to state $s^*$ in one step has a positive probability under some $(s,a)$, then the transition occurs to the unique state $z\neq s^*$ such that $p(\{z,s^*\}|s,a)=1$. Notice that in the MDP $\mathcal{M}'$, each action in each state leads to a specific state with probability 1.

Let $\sigma$ be a pure Blackwell optimal stationary strategy in $\calM'$. 
We will show that $\sigma$ satisfies the condition of Theorem \ref{theorem-det}, that is, $\sigma$ is not weakly overtaken by any other pure strategy.

Consider any other pure strategy $\rho$ in $\calM'$. 
Since the MDP $\calM'$ is deterministic,
each of the pure strategies $\sigma$ and $\rho$ induces a specific infinite history in $\mathcal{M}'$ with probability 1. Let $(s_t,a_t)_{t\in\NN}$ denote the infinite history induced by $\sigma$ and $(z_t,b_t)_{t\in\NN}$ denote the infinite history induced by $\rho$.

In the original MDP $\mathcal{M}$, the probability under $\sigma$ that state $s^*$ is not reached within the first $t$ periods is $1-\PP_{s_1,\sigma}(t^*\leq t)$. This probability is related to the payoffs in the average payoff MDP $\mathcal{M}'$ as follows. For each period $t\geq 2$ we have
\begin{align*}
\log\left[1-\PP_{s_1,\sigma}(t^*\leq t)\right] &\;=\; \log\left[\prod_{k=1}^{t-1}p(s_{k+1} | s_k,a_k) \right] \\
&\;=\; \sum_{k=1}^{t-1}\log\left[p(s_{k+1} | s_k,a_k)\right] \\
&\;=\; -\,\sum_{k=1}^{t-1}u'(s_k,a_k)\\
&\;=\; -(t-1)\cdot u_{t-1}(s_1,\sigma).
\end{align*}
Similarly, we have
\[
\log\left[1-\PP_{s_1,\rho}(t^*\leq t)\right] \;=\; -\,\sum_{k=1}^{t-1}u'(z_k,b_k)\;=\;-(t-1)\cdot u_{t-1}(s_1,\rho).
\]

By the choice of $\sigma$ in $\mathcal{M}'$ (cf. Theorem \ref{extended spor}), one of the following holds: (i) There is $M\in\NN$ such that for all periods $t\geq M$ we have $u_t(s_1,\sigma)\,=\,u_t(s_1,\rho)$. (ii) There is a strictly increasing sequence $(t_k)_{k\in\NN}$ of periods such that for each $k\in\NN$ we have $u_{t_k}(s_1,\sigma)\,>\,u_{t_k}(s_1,\rho)$.

If (i) holds, then $\PP_{s_1,\sigma}(t^*\leq t) \,=\, \PP_{s_1,\rho}(t^*\leq t)$ for all $t\geq M+1$, and hence $\rho$ does not weakly overtake $\sigma$ in the original MDP $\mathcal{M}$.

If (ii) holds, then $\PP_{s_1,\sigma}(t^*\leq t_k) \,>\, \PP_{s_1,\rho}(t^*\leq t_k)$ for each $k\in\NN$, and hence $\rho$ does not weakly overtake $\sigma$ in the original MDP $\mathcal{M}$ in this case either.
\end{proof}

The following example demonstrates that, even if the MDP with a reachability objective is deterministic, an overtaking optimal strategy may fail to exist. In the example, each pure strategy is equally good in the overtaking sense, but each pure strategy is overtaken by any strategy that uses randomization at every period.

\begin{example}\label{example: 2}
Consider the MDP with a reachability objective
that is depicted in Figure~\ref{fig1}, 
where the notation is similar to that of Example \ref{example: 1}. The initial state is state $x$.

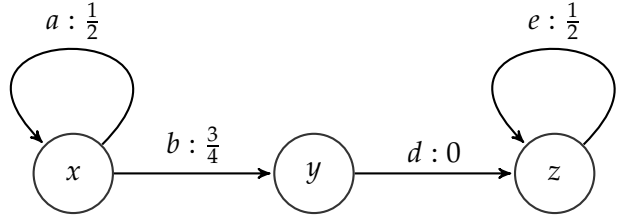
\begin{figure}[H]
\label{fig1}
\begin{center}
\begin{tikzpicture}[scale=0.8,>=stealth',shorten >=1pt,auto,node distance=2.8cm]
\node at(0,3) [state] (s1) {$x$};
\node at(4,3) [state] (s2) {$y$};
\node at(8,3) [state] (s4) {$z$};
\path[->,thick]
 (s1) edge[loop] node[above] {$a:\frac{1}{2}$} ()
 (s1) edge node[above] {$b:\frac{3}{4}$} (s2)
 (s2) edge node[above] {$d:0$}(s4)
 (s4) edge[loop] node[above] {$e:\frac{1}{2}$} ();
\end{tikzpicture}
\end{center}
\caption{The MDP in Example~\ref{fig1}.}
\end{figure}

When playing action $b$ and subsequently action $d$, 
the accumulative probability during these two periods of reaching the target state is $\frac{3}{4}$. 
Playing action $a$ twice (or action $e$ twice) 
leads to the same accumulative probability, as $\frac{1}{2}+\frac{1}{2}\cdot \frac{1}{2}=\frac{3}{4}$. 
It follows that $\PP_{x\sigma}(t^* \leq t) = 1-\tfrac{1}{2^{t-1}}$,
for all pure strategies $\sigma$, 
except of the strategy $\sigma' = a^{t-2}b$ that plays the action $a$ in the first $t-2$ periods
and the action $b$ in period $t-1$, for which $\PP_{x\sigma'}(t^* \leq t) > 1 - \tfrac{1}{2^{t-1}}$.
This will imply that strategies that use randomization at state $x$ at every period can do better
in the overtaking sense than all pure strategies.
Indeed, given any period $t\geq 2$, when calculating the probability of reaching the target state within the first $t$ periods, there is a positive probability that action $b$ is played exactly at period $t-1$ (and thus we reach the target state exactly at the last period $t$), while not having to include consequences of playing action $d$ yet.

\begin{claim} Consider Example \ref{example: 2}. For each pure strategy $\sigma$, it holds for sufficiently large periods $t$ that the probability of reaching the target state within the first $t$ periods is
\[\PP_{x\sigma}(t^*\leq t)\;=\;1-\frac{1}{2^{t-1}}.\]
In particular, no pure strategy is overtaken by another pure strategy.
\end{claim}

\begin{proof}
For the pure strategy $a^\infty$ that plays $a$ at all periods, we have for all periods $t$ that
\[\PP_{xa^\infty}(t^*\leq t)\;=\;1-\frac{1}{2^{t-1}}.\]
For any other pure strategy $a^{n-1}b$ that plays $a$ at the first $n-1$ periods and plays $b$ at period $n$, we have for all periods $t\geq n+1$ that
\[\PP_{x,a^{n-1}b}(t^*\leq t)\;=\;1-\frac{1}{2^{n-2}}\cdot \frac{1}{4}\cdot 1\cdot \frac{1}{2^{t-n-1}}\;=\;1-\frac{1}{2^{t-1}}.\]
This completes the proof of the claim.
\end{proof}

\begin{claim}\label{laterprob} Consider Example \ref{example: 2}. Take two strategies $\sigma$ and $\sigma'$. Consider a period $t\geq 2$. Then, $\PP_{x\sigma}(t^*\leq t)>\PP_{x\sigma'}(t^*\leq t)$ holds if and only if the probability under $\sigma$ of being in state $x$ and playing action $b$ at period $t-1$ is strictly larger than that under $\sigma'$, 
i.e., $\PP_{x\sigma}(a_{t-1}=b)> \PP_{x\sigma'}(a_{t-1}=b)$.

Consequently, if the condition $\PP_{x\sigma}(a_{t-1}=b)> \PP_{x\sigma'}(a_{t-1}=b)$ 
holds for all sufficiently large periods $t$, then $\sigma$ overtakes $\sigma'$.
\end{claim}

\begin{proof}
Suppose that when playing two strategies $\sigma$ and $\sigma'$, 
it holds that for some period $t\geq 2$ we have $\PP_{x\sigma}(a_{t-1}=b)> \PP_{x\sigma'}(a_{t-1}=b)$.

On the finite horizon up to period $t$, the set of pure strategies is $W^t=\{a^{t},\ b,\ ab,\ a^2b,\ldots,\ a^{t-1}b\}$. Under the pure strategy $a^{t-2}b$, the probability of reaching the target state within the first $t$ periods is
\[\PP_{x,a^{t-2}b}(t^*\leq t)\;=\;1-\frac{1}{2^{t-2}}\cdot \frac{1}{4}\;=\;1-\frac{1}{2^t}.\]
Under each other pure strategy $\tau\neq a^{t-1}b$, this probability is
\[\PP_{x\tau}(t^*\leq t)\;=\;1-\frac{1}{2^{t-1}}.\]

By Kuhn's theorem, each of the strategies $\sigma$ and $\sigma'$ can be seen as a probability distribution over the set of pure strategies $W^t$. Hence, we have
\begin{align*}
\PP_{x\sigma}(t^*\leq t) &\;=\;\sum_{\tau\in W^t}\PP_{x\sigma}(\tau)\cdot \PP_{x\tau}(t^*\leq t)\\
&\;=\;\PP_{x\sigma}(a^{t-2}b)\cdot \Big(1-\frac{1}{2^t}\Big)+\;(1-\PP_{x\sigma}(a^{t-2}b))\cdot \Big(1-\frac{1}{2^{t-1}}\Big),
\end{align*}
and similarly for the strategy $\sigma'$.

Thus, $\PP_{x\sigma}(t^*\leq t)>\PP_{x\sigma'}(t^*\leq t)$ holds if and only if $\PP_{x\sigma}(a^{t-2}b)>\PP_{x\sigma'}(a^{t-2}b)$, 
which is further equivalent with $\PP_{x\sigma}(a_{t-1}=b)> \PP_{x\sigma'}(a_{t-1}=b)$. 
The proof is complete.
\end{proof}

By Claim \ref{laterprob} above, the stationary strategy $(\frac{1}{2},\frac{1}{2})^\infty$ that always chooses action $a$ and action $b$ each with probability $\frac{1}{2}$ overtakes every pure strategy. Also, the stationary strategy $(p,1-p)^\infty$ overtakes the stationary strategy $(q,1-q)^\infty$ if $q<p<1$, as for large periods $t$ we have
\[\PP_{x,(p,1-p)^\infty}(a_{t-1}=b)\;=\;p^{t-2}(1-p)\;>\;\;q^{t-2}(1-q)\;=\; \PP_{x,(q,1-q)^\infty}(a_{t-1}=b).\]
This already shows that in Example~\ref{example: 2} there is no stationary overtaking optimal strategy. 
We now show that there is no overtaking optimal strategy at all.

\begin{claim} 
In the MDP in Example \ref{example: 2} there exists no overtaking optimal strategy.
\end{claim}

\begin{proof}
Consider any strategy $\sigma$. We construct a strategy that overtakes $\sigma$. 
The strategy $\sigma$ can be seen as a sequence $(z_n)_{n=1}^\infty$ where $z_n$ denotes the probability that $\sigma$ assigns to action $b$ when being in state $x$ at period $n$. We distinguish two cases.

Case 1: Assume that either $z_n=0$ for all periods $n$ or $z_n=1$ for some period $n$. 
In this case, $\PP_{x\sigma}(a_n=b)=0$ at large periods $n$. Hence, by Claim \ref{laterprob}, the stationary strategy $(\frac{1}{2},\frac{1}{2})^\infty$ overtakes $\sigma$.

Case 2: Assume that $z_m>0$ for some period $m$ and $z_n<1$ for all periods $n$. We can choose a sequence $(z'_n)_{n=1}^\infty$ such that (i) for all periods $n=1,\ldots,m-1$ we have $z'_n=z_n$, (ii) for period $m$ we have $z'_m<z_m$, (iii) for all periods $n>m$ we have $z_n<z'_n<1$, and (iv) $\prod_{n=1}^\infty (1-z'_n)=\prod_{n=1}^\infty (1-z_n)$. The idea is to slightly reduce the probability $z_m$ at period $m$ and slightly increase all probabilities $z_n$, $n>m$, so that (iv) holds, i.e., the total probability of ever playing $b$ under $(z_n)_{n=1}^\infty$ is equal to that under $(z'_n)_{n=1}^\infty$.

Let $\sigma'$ be the strategy corresponding to $(z'_n)_{n=1}^\infty$. Consider any period $t>m$. By (iii), we have $\prod_{n=t}^\infty (1-z'_n)\leq \prod_{n=t}^\infty (1-z_n)$. Hence, by (iv), we obtain $\prod_{n=1}^{t-1} (1-z'_n)\geq \prod_{n=1}^{t-1} (1-z_n)$. This means that the probability of being in state $x$ at period $t$ is at least as large under $\sigma'$ as under $\sigma$. 
Thus, by (iii), we obtain $\PP_{x\sigma'}(a_t=b)> \PP_{x\sigma}(a_t=b)$. Since this is true for all periods $t>m$, in view of Claim \ref{laterprob}, $\sigma'$ overtakes $\sigma$.
\end{proof}
\end{example}

\section{Reachability objectives: Non-deterministic MDPs}

The following example,
which is an adaptation of Example~\ref{example: 2} to the context of non-deterministic MDPs, 
shows that, in non-deterministic MDPs with a reachability objective, 
it can happen that each pure strategy is overtaken by another pure strategy. 
As a consequence, Theorem \ref{theorem-det} cannot be extended to non-deterministic MDPs.\\

\begin{example}\label{example each pure}
Consider the MDP with initial state $x$ and a reachability objective that is depicted in Figure~\ref{fig3}.
In this MDP, the only choice of the decision maker is when to play action $c$, 
if at all, and action $c$ can be played at most once.

\begin{figure}[H]
\label{fig3}
\begin{center}
\begin{tikzpicture}[scale=0.8,>=stealth',shorten >=1pt,auto,node distance=2.8cm]
\node at(0,3) [state] (s1) {$x$};
\node at(8,3) [state] (s2) {$y$};
\node at(4,-1) [state] (s5) {$x'$};
\node at(4,3) (s3){};
\node at(8,-1) (s6){};
\node at(12,3) [state] (s4) {$z$};
\path[->,thick]
 (s1) edge[loop] node[above] {$a:\frac{1}{2}$} ()
 (s1) edge node[above] {$c$}(s3)
 (s2) edge node[above] {$d:0$}(s4)
 (s3) edge node[above] {$(Prob.\, \frac{1}{2}):\frac{3}{4}$} (s2)
 (s4) edge[loop] node[above] {$e:\frac{1}{2}$} ()
 (s3)  edge  node[left] {$(Prob.\, \frac{1}{2}):\frac{1}{2}$} (s5)
 (s5) edge node[above] {$c'$} (s6)
 (s6) edge node[right] {$(Prob.\, \frac{1}{2}):\frac{3}{4}$} (s2)
 (s6)  edge   [bend left=40]   node[below] {$(Prob.\, \frac{1}{2}):\frac{1}{2}$} (s5);
\end{tikzpicture}
\end{center}
\caption{The MDP in Example~\ref{example each pure}.}
\end{figure}
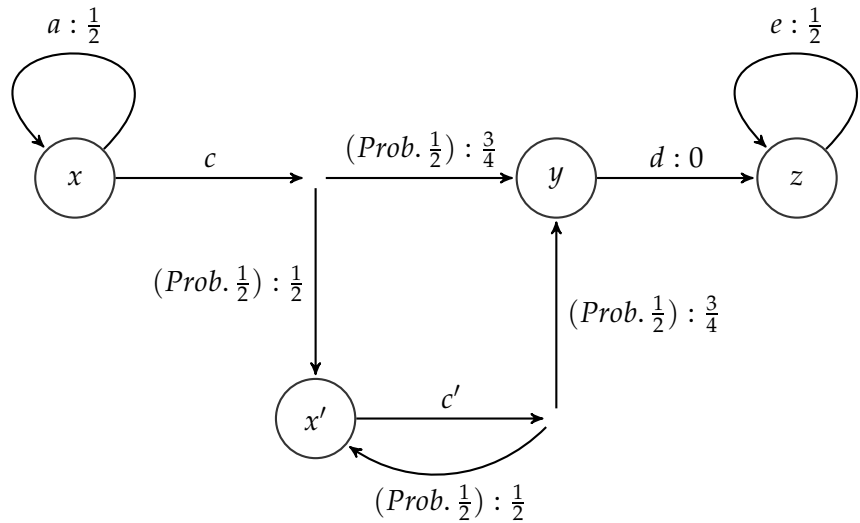

\noindent

In this MDP, action $c$ leads to the target state $s^*$ with probability $5/8$, to state $y$ with probability $1/8$ and to state $x'$ with probability $1/4$. It will be easier to think about action $c$ in the following way, which gives the same transition probabilities: After playing action $c$, a lottery is executed: (1) with probability $1/2$ the play follows the upper-part of the arrow, and thus the play moves to state $s^*$ with probability $3/4$ and to state $y$ with probability $1/8$, and (2) with probability $1/2$ the play follows the bottom-part of the arrow, and thus the play moves to state $s^*$ with probability $1/2$ and to state $x'$ with probability $1/2$. Action $c'$ in state $x'$ has a similar interpretation.

The pure strategies in this MDP are $a^\infty$, $c$, $ac$, $a^2c$, $\ldots$. The strategy $a^\infty$ corresponds to strategy $a^\infty$ in Example \ref{example: 2}, and the strategy $a^tc$ corresponds to the mixed strategy in Example \ref{example: 2} that, in state $x$, recommends action $a$ up to period $t$ and the mixed action $(\frac{1}{2},\frac{1}{2})$ at all periods after $t$. The reader can verify that the strategy $a^\infty$ is overtaken by the strategy $c$, and each strategy $a^tc$ is overtaken by the strategy $a^{t+1}c$. That is, 
each pure strategy is overtaken by another pure strategy.
\end{example}

Next, we present sufficient conditions for the existence of a strongly overtaking optimal strategy.

\begin{theorem}\label{theorem non deter}
Consider an MDP with a reachability objective. 
Suppose that, from every initial state, there is a strategy that weakly overtakes all other strategies. 
Then, there exists a stationary strategy that is strongly overtaking optimal for each initial state.
\end{theorem}

\begin{proof}
Note that if strategy $\sigma$ weakly overtakes all other strategies,
then $\sigma$ is the unique strategy with this property.

For each $s \in S$, let $\sigma_s$ denote the unique strategy that weakly overtakes all other strategies
at the initial state $s$, 
and let $\alpha_s$ denote the mixed action by $\sigma_s$ at the initial state $s$.

Let $\alpha$ be the stationary strategy that uses the mixed action $\alpha_s$ at state $s$, for all $s\in S$. 
Our goal is to show that $\alpha$ is strongly overtaking optimal for each initial state.

Fix an initial state $s \in S$. 
For each state $z$, let $H_s(z)$ denote the set of histories $h$ such that 
(1) $h$ has a positive probability under $\sigma_s$ from initial state $s$, and 
(2) $h$ ends in state $z$. We show that $\sigma_s(h)=\alpha_z$. 
Indeed, suppose by way of contradiction that $\sigma_s(h)\neq \alpha_z$. 
Let $\sigma'_s$ be the strategy such that 
(i) $\sigma'_s$ follows $\sigma_s$ outside the subgame at $h$, and 
(ii) in the subgame at $h$, the continuation strategy $\sigma[h]$ is replaced by $\sigma_z$. 
Then, for all $t\geq ||h||$
\[\PP_{\sigma'_s}(t^*\leq t)-\PP_{\sigma_s}(t^*\leq t)\,=\,\PP_{\sigma_s}(h)\cdot \left[\PP_{z,\sigma_z}(t^*\leq t)-\PP_{z,\sigma[h]}(t^*\leq t)\right].\]
Since $\sigma_z$ weakly overtakes $\sigma[h]$ for initial state $z$, 
we can conclude that $\PP_{z,\sigma_z}(t^*\leq t)-\PP_{z,\sigma[h]}(t^*\leq t)$ is non-negative for all large $t$ and strictly positive for infinitely many $t$. 
Thus, the same holds for $\PP_{\sigma'_s}(t^*\leq t)-\PP_{\sigma_s}(t^*\leq t)$, 
and hence $\sigma_s$ is weakly overtaken by $\sigma'_s$. This is a contradiction.

Hence, for each initial state $s$, each history $h$ has the same probability under the strategy $\sigma_s$ and under the stationary strategy $\alpha$. 
Using that $\sigma_s$ weakly overtakes all other strategies, $\sigma_s$ is strongly overtaking optimal, 
and hence so is $\alpha$. (In fact, $\sigma_s=\alpha$ must hold.)
\end{proof}

\section{Reachability objectives: Generic MDPs}\label{sect-gen}

In this section we consider generic MDPs with a reachability objective. The main result of this section is the following theorem.

\begin{theorem}\label{gen-thm}
In generic MDPs with a reachability objective, 
there is a pure stationary strategy that overtakes each other stationary strategy at each initial state.
\end{theorem}

Let the state space be $S=\{1,\ldots,n\}$, where $n\geq 2$, and let the target state be state $s^*=n$. We can assume without loss of generality that state $n$ is absorbing. 

In this section we assume that the transition probabilities are randomized: for each state $s\in S\setminus\{n\}$ and action $a\in A(s)$, the probability distribution of the next state $(p(z|s,a))_{z\in S}$ is drawn independently from some continuous distribution on $\{(q_1,\ldots,q_n)\in (0,1)^n\,|\,q_1+\cdots+q_n=1\}$.

The idea of the proof of Theorem \ref{gen-thm} is as follows. Every stationary strategy $\sigma$ defines a transition matrix $A_\sigma$. Under the strategy $\sigma$, the rate of absorption to state $n$ is exactly $\lambda_2(A_\sigma)$, the second largest eigenvalue of $A_\sigma$. Thus, if $\lambda_2(A_\sigma) < \lambda_2(A_{\sigma'})$, then $\sigma$ overtakes $\sigma'$ 
for the reachability objective.
Let $A_1, A_2, \cdots, A_K$ be all transition matrices that are induced by pure stationary strategies.
Generically, the second largest eigenvalues of these matrices differ,
and therefore there is one of them, say, $A_1$, whose second largest eigenvalue is minimal.
The matrix $A_\sigma$ is in the convex hull of the matrices $A_1, A_2, \cdots, A_K$,
and we will prove that if $A_\sigma\neq A_1$ then $\lambda_2(A_\sigma) > \lambda_2(A_1)$.
This will imply that the pure stationary strategy that corresponds to the matrix $A_1$ overtakes each other stationary strategy at each initial state.

The proof of Theorem~\ref{gen-thm} consists of four steps.

\begin{itemize}
\item[Step 1:] Proving that the second largest eigenvalue determines the overtaking relation between stationary strategies: 
When comparing two stationary strategies $\sigma_A$ and $\sigma_B$ 
generating the respective transitions matrices $A$ and $B$, 
$\lambda_2(A)<\lambda_2(B)$ implies that $\sigma_A$ overtakes $\sigma_B$ at each initial state.

\item[Step 2:]  Proving that the second largest eigenvalue of a transition matrix $A$, corresponding to a stationary strategy, is equal to the largest eigenvalue of the $(n-1)\times (n-1)$ submatrix $A'$ that remains when from $A$ we remove the column and the row associated with the target state: $\lambda_2(A) = \lambda_1(A')$.

\item[Step 3:] Proving that for two positive square matrices $A$ and $B$ that differ only in one row, 
the largest eigenvalue of any convex combination of the two matrices cannot be lower 
than the minimum between the largest eigenvalues of the two matrices:
for every $\alpha\in(0,1)$ we have $\lambda_1(\alpha A +(1-\alpha)B)\geq \text{min}\left\{\lambda_1(A), \lambda_1(B)\right\}$, and if $\lambda(A)\neq\lambda(B)$ then the inequality is strict.

\item[Step 4:] Proving that it is sufficient to consider only matrices that differ in one row. 
This implies that the pure stationary strategy that corresponds to the transition matrix
with minimal second largest eigenvalue overtakes each other stationary strategy, at each initial state.
\end{itemize}

\bigskip

\noindent\textbf{Proof of Step 1:} The statement of Step 1 follows from the combination of Theorem~4 and Theorem~6 in  De Santis and Spizzichino (2012).\footnote{These two theorems in De Santis and Spizzichino (2012) imply for every initial state $s=1,\ldots,n-1$ that $\mathbb{P}_{s,\sigma_A}(t^*\leq t)\,\geq \,\mathbb{P}_{s,\sigma_B}(t^*\leq t)$ for large $t\in\NN$. However, their proof can be easily adapted to show that $\mathbb{P}_{s,\sigma_A}(t^*\leq t)\,> \,\mathbb{P}_{s,\sigma_B}(t^*\leq t)$ for large $t\in\NN$, so that strategy $\sigma_A$ overtakes strategy $\sigma_B$. Indeed, the inequalities (37) and (38) on page 12 do not only imply $\sum_{j=0}^\ell \widetilde{p}^{(n)}_{i,j}\,\geq \sum_{j=0}^\ell p^{(n)}_{i,j}$, but they even imply the strict inequality.}
\bigskip

\noindent\textbf{Proof of Step 2:} Let $\sigma$ be a stationary strategy with transition matrix 
\[A=\begin{bmatrix}
 x_{1,1}       & x_{1,2} & x_{1,3} & \dots & x_{1,n} \\
x_{2,1}       & x_{2,2} & x_{2,3} & \dots & x_{2,n} \\
 \vdots & \vdots & \vdots & \ddots & \vdots \\
 x_{n-1,1}       & x_{n-1,2} & x_{n-1,3} & \dots & x_{n-1,n}\\
 0 & 0 & 0 & \dots & 1
 \end{bmatrix},\]
with entry $(i,j)$ being the probability under $\sigma$ of moving to state $j$ from state $i$. Since the transition probabilities of the MDP are randomized, it holds generically that $x_{i,j}>0$ for all $i=1,\ldots,n-1$ and $j=1,\ldots,n$. The sum of entries in each row of the matrix $A$ is equal to 1, and hence the largest eigenvalue of $A$ is $1$, with eigenvector $(0,0,\ldots,0,1)$.

Consider also the submatrix of $A$ that arises when we delete the last column and the last row (which correspond to the target state):
\[A'=\begin{bmatrix}
 x_{1,1}       & x_{1,2} & x_{1,3} & \dots & x_{1,n-1} \\
x_{2,1}       & x_{2,2} & x_{2,3} & \dots & x_{2,n-1} \\
 \vdots & \vdots & \vdots & \ddots & \vdots \\
 x_{n-1,1}       & x_{n-1,2} & x_{n-1,3} & \dots & x_{n-1,n-1}
 \end{bmatrix}.\]
According to the Perron-Frobenius Theorem, the largest eigenvalue $\lambda_1(A')$ of the matrix $A'$ is a real number. As the sum of entries in each row of $A'$ is strictly less than 1, we have $\lambda_1(A')<1$.

The proof that $\lambda_1(A')$ is the second largest eigenvalue of the matrix $A$ follows from the following two observations:

(i) Any eigenvalue of $A'$ is also an eigenvalue of $A$. Indeed, let $\mu$ be an eigenvalue of $A'$ with right eigenvector $(y_1,\ldots,y_{n-1})$.\footnote{Recall that the same set of eigenvalues correspond both to right and left eigenvectors.}  Then $\mu$ is an eigenvalue of $A$ with right eigenvector $(y_1,\ldots,y_{n-1},0)$.

(ii) If $\mu\neq 1$ be an eigenvalue of $A$, then $\mu$ is also an eigenvalue of $A'$. Indeed, let $y=(y_1,\ldots,y_n)$ be a right eigenvector of $A$ corresponding to $\mu$. Then, $Ay=\mu y$. This implies $y_n=\mu\cdot y_n$, which is only possible if $y_n=0$. Hence, $\mu$ is an eigenvalue of $A'$ with eigenvector $(y_1,\ldots,y_{n-1})$.\bigskip

\noindent\textbf{Proof of Step 3:} The statement of Step 3 follows from the following theorem.

\begin{theorem}\label{them eigen}
Let $A$ and $B$ be two positive (all elements are positive) square matrices that differ only in the first row.
For every $\alpha \in (0,1)$ define $M_\alpha := \alpha A + (1-\alpha)B$.
We will prove that $\lambda_{1}(M_{\alpha})\geq \text{min}\left\{\lambda_1(A),\lambda_1(B)\right\}$,
and, if $\lambda_1(A)\neq\lambda_1(B)$, then $\lambda_{1}(M_{\alpha})> \text{min}\left\{\lambda_1(A),\lambda_1(B)\right\}$. 

\end{theorem}

\begin{remark}
If the two matrices $A$ and $B$ differ in several rows, 
then Theorem~\ref{them eigen} is no longer true. To see that, consider the following matrices:
\[A= \begin{bmatrix}
 \tfrac{98}{300} & \tfrac{98}{300} &\tfrac{1}{300} \\
\tfrac{98}{300}  & \tfrac{1}{300} &\tfrac{1}{300} \\
\tfrac{1}{300} &\tfrac{1}{300} & \tfrac{1}{300}
 \end{bmatrix},\hspace{0.6cm} \ \ B= \begin{bmatrix}
 \tfrac{1}{300} & \tfrac{1}{300} &\tfrac{1}{300}  \\
\tfrac{1}{300} & \tfrac{1}{300} &\tfrac{98}{300} \\
\tfrac{1}{300} &\tfrac{98}{300} &\tfrac{98}{300}
 \end{bmatrix}.\]
\noindent 
For these matrices, $\lambda_1(A)=\lambda_1(B)=0.529522$, 
while $\lambda_1\left(\tfrac{1}{2}A + \tfrac{1}{2}B\right)=\tfrac{1}{3} < 0.529522$.
\end{remark}

We present two proofs for Theorem~\ref{them eigen}.%
\footnote{We thank Robert Israel for suggesting the first proof, 
and Omri Solan for suggesting the second proof.}

\begin{proof}[First Proof] 
Denote $\lambda^*=\lambda_1(M_{\alpha})$.

By the Perron-Frobenius theorem, there exists a positive right eigenvector $u = (u_j)_j$ corresponding to $\lambda^*$ for the matrix $M_{\alpha}$.
Let $e_j$ be the $j$'th unit vector.
Note that $u_j$ is the $j$'th coordinate of $u$, while $e_j$ is the $j$'th unit vector.
For every $j \neq 1$ we have
$e_j^T A  = e_j^T B$, hence $e_j^T A u = e_j^T B u$.
Since
\[\lambda^* u_j = \lambda^* e_j^T u = e_j^T M_{\alpha} u = \alpha e_j^T A u +  (1-\alpha)e_j^T B u, \]
we deduce that
\begin{equation}
\label{equ:10}
e_j^T A u = e_j^T B u = \lambda^* u_j, \ \ \ \forall j \neq 1.
\end{equation}
For $j=1$ we have
\[\lambda^* u_1 = \lambda^* e_1^T u = e_1^T M_{\alpha} u = \alpha e_1^T A u +  (1-\alpha)e_1^T B u. \]
Assume w.l.o.g. that $e_1^T A u \geq e_1^T B u$, so that

\begin{equation}
\label{equ: 10b}
 e_1^T A u \geq \lambda^* u_1 \geq e_1^T B u. \end{equation}
It follows from Eqs.~\eqref{equ:10} and~\eqref{equ: 10b} that for every nonnegative vector $v$ we have
\begin{equation}
\label{equ:11}
v^T A u \geq \lambda^* v^T u \geq v^T B u.
\end{equation}
By the Perron-Frobenius theorem, there exists a positive right eigenvector $v_A$ for $\lambda_1(A)$,
and a positive right eigenvector $v_B$ for $\lambda_1(B)$.
Substituting $v=v_A$ in the left-hand side of Ineq.~\eqref{equ:11} and
substituting $v=v_B$ in the right-hand side of Ineq.~\eqref{equ:11} we obtain
\begin{equation}
\label{equ:12}
\lambda_1(A) v_A^T u = v_A^T A u \geq \lambda^* v^T_A u,
\end{equation}
and
\begin{equation}
\label{equ:13}
\lambda_1(B) v_B^T u = v_B^T B u \leq \lambda^* v^T_B u.
\end{equation}
Since $v^T_A u$ and $v^T_B u$ are positive reals, this implies that $\lambda_1(A) \geq \lambda^* \geq \lambda_1(B)$.

Suppose now that $\lambda_1(A) > \lambda_1(B)$.
Then necessarily $e_1^T A u > e_1^T B u$, since, if the two are equal, then there would have been equality in Eq.~\eqref{equ:11},
and then we would obtain that $\lambda_1(A) = \lambda_1(B)$.
As all coordinates of $v_A$ and $v_B$ are positive,
and in particular the first coordinates are positive,
there is a strict inequality in Eqs.~\eqref{equ:12} and~\eqref{equ:13}.
\end{proof}

\begin{proof}[Second Proof]
Assume to the contrary that the function $\alpha \mapsto \lambda_1(M_\alpha)$ is not monotone.
Then there exist $0 < \alpha_1 < \alpha_2 < \alpha_3 < \alpha_4 < 1$ such that 
$\lambda_1(M_{\alpha_2}) < \lambda_1(M_{\alpha_1}) = \lambda_1(M_{\alpha_3}) < \lambda_1(M_{\alpha_4})$.
Set $\lambda = \lambda_1(M_{\alpha_1}) = \lambda_1(M_{\alpha_3})$.

For every $x \in \dR^n_{++}$ denote by $A_x$ the matrix that coincides with $A$, 
except that its first row is $x$.
Denote by $\calM_A := \{A_x \colon x \in \dR^n_{++}\}$ the space of all these matrices.
The space $\calM_A$ is equivalent to $\dR^n_{++}$.

By the Perron-Frobenius Theorem the largest eigenvalue of a positive matrix is a positive real.
Denote by $f : \dR^n_{++} \to \dR$ the function that assigns to each $x \in \dR^n_{++}$ 
the largest eigenvalue of the matrix $A_x$, that is, $f(x)=\lambda_1(A_x)$.
Since the function that maps each $x$ to the norm of the largest eigenvalue is continuous,
the function $f$ is continuous.

Denote by $H_\lambda = \{ x \colon \det(\lambda I - A_x) = 0\} \subseteq \dR^n_{++}$ 
the set of all matrices that have $\lambda$ as an eigenvalue.
The function $x \mapsto \det(\lambda I - A_x)$ is linear, 
hence the set $H_\lambda$ is the intersection of a hyperspace with $\dR^n_{++}$.
Plainly, $\{ x \colon f(x) = \lambda\} \subseteq H_\lambda$.
We will show that $\{ x \colon f(x) = \lambda\} = H_\lambda$. 

On the one hand, the complement of the set  $\{ x \colon f(x) = \lambda\}$ 
is the union of the two sets $\{ x \colon f(x) > \lambda\}$ and $\{ x \colon f(x) < \lambda\}$.
Note that these two sets are non-empty, because 
$\lambda_1(M_{\alpha_2}) < \lambda$ and $\lambda_1(M_{\alpha_4}) > \lambda$.
The continuity of $f$ implies that the set $\{x \colon f(x) \neq \lambda\}$ is disconnected.

On the other hand, the complement of the set $\{ x \colon f(x) = \lambda\}$ 
is the union of the 
two half-spaces $\{ x \colon \det(\lambda I - A_x) > 0\}$ and $\{ x \colon \det(\lambda I - A_x) < 0\}$,
and of the set $H_\lambda \setminus \{ x \colon f(x) = \lambda\}$.
If the latter set were nonempty, then the union of the three sets would have been connected,
a contradiction.

It follows that $\{ x \colon f(x) = \lambda\} = H_\lambda$.
This in particular implies that 
$\lambda_1(M_{\alpha_1}) = \lambda_1(M_{\alpha_2}) = \lambda_1(M_{\alpha_3}) = \lambda_1(M_{\alpha_4})$,
which contradicts the choice of $\alpha_1, \alpha_2, \alpha_3, \alpha_4$.
Consequently, if $\lambda_1(A) = \lambda_1(B)$
then the function $\alpha \mapsto \lambda_1(M_\alpha)$ is constant,
while if $\lambda_1(A) \neq \lambda_1(B)$ then the function $\alpha \mapsto \lambda_1(M_\alpha)$
is strictly monotone. Observe that this result is slightly stronger than the statement of the theorem.
\end{proof}

\bigskip

\noindent\textbf{Proof of Step 4:} Let $A_1, A_2, \cdots, A_K$ be all transition matrices that are induced by pure stationary strategies. Generically, the second largest eigenvalues of these matrices are all different. Assume that $\lambda_2(A_1)<\lambda_2(A_i)$ for all $i=2,\ldots,n$. Let $\sigma$ be the pure stationary strategy corresponding to the transition matrix $A_1$.

Let $\tau$ be a stationary strategy, and let $A$ be the transition matrix corresponding to $\tau$. For every $r=0,1,\ldots,n-1$ let $\calB_r$ be the collection of all matrices that
coincide with $A$ in the first $r$ rows, 
and coincide with one of the matrices $A_1, A_2, \cdots, A_K$ in the other $n-r$ rows.
Note that $\calB_0 = \{A_1, A_2, \cdots, A_K\}$. Using Step 2 together with Step 3 inductively we obtain that
\begin{equation}
\label{equ:89}
 \lambda_2(A) \geq \min_{B \in \calB_{n-1}} \lambda_2(B) \geq 
\min_{B \in \calB_{n-2}} \lambda_2(B) \geq \cdots \geq \min_{B \in \calB_{0}} \lambda_2(B). 
\end{equation}
Moreover, if $A \notin \calB_{0}$, then at least one of the inequalities in Eq.~\eqref{equ:89} is strict. 

Now assume that $\tau\neq \sigma$. If $A \notin \calB_{0}$ then $\lambda_2(A)>\min_{B \in \calB_{0}} \lambda_2(B)=\lambda_2(A_1)$, whereas if $A \in \calB_{0}$ then $\lambda_2(A)>\lambda_2(A_1)$ by the choice of $A_1$. Thus, by Step 1, $\sigma$ overtakes $\tau$ at each initial state.\hfill $\Box$

\begin{remark} Consider a generic MDP with a reachability objective. According to Theorem \ref{gen-thm}, there is a pure stationary strategy $\sigma$ that overtakes each other stationary strategy at each initial state. Given any stationary strategy $\sigma'\neq \sigma$ and initial state $s$, one can compute a horizon $T$ such that $\sigma$ outperforms $\sigma'$ beyond $T$: for all $t\geq T$ we have $\PP_{s\sigma}(t^*\leq t)> \PP_{s\sigma'}(t^*\leq t)$. For the details we refer to the Appendix. 
\end{remark}

\section{Safety objectives}

In this section we consider MDPs with a safety objective:
the decision maker's objective is to avoid the state $s^*$ with a probability as high as possible.\medskip

\noindent\textbf{The statement of Theorem~\ref{theorem-det}:} 
The statement remains valid for safety objectives. The proof requires the following changes:

(1) For the MDP $\mathcal{M}$, 
we can still assume without loss of generality that there is no state $s\neq s^*$ 
and action $a\in A(s)$ such that $p(s^*|s,a)=1$, but the reason is different. 
Such an action is the worst for the decision maker, so it can be deleted. If all actions in a state $s$ would be deleted, then we can delete the state $s$ and replace each transition to state $s$ by a transition directly to state $s^*$.

(2) The payoffs in the auxiliary MDP $\mathcal{M}'$ is defined to be the opposite
of the payoff in the proof of Theorem~\ref{theorem-det}: 
if $z\neq s^*$ denotes the unique state for which $p(\{z,s^*\}|s,a)=1$, 
then the payoff is $u'(s,a)=\log(p(z|s,a))$. 
The reason is that, now with a safety objective, the decision maker prefers that $p(z|s,a)$ is large.\medskip

\noindent\textbf{The statements  of Theorem~\ref{theorem non deter} and Theorem~\ref{gen-thm}:} The statements remain valid for safety objectives, with proofs that are almost identical to the ones for reachability objectives.\medskip

\noindent\textbf{The counter-part of Example \ref{example: 2}:} 
Also for safety objectives, we can construct an MDP in which the decision maker has no 
overtaking optimal strategy. Indeed, take the MDP in Example \ref{example: 2}, 
and replace $b:\frac{3}{4}$ with $b:0$ and replace $d:0$ with $d:\frac{3}{4}$. 
That is, now action $b$ leads to state $s^*$ with probability 0 and action $d$ leads to state $s^*$ 
with probability $\frac{3}{4}$. 
In this MDP with the safety objective, by comparing the transition probabilities, action $b$ is still to be preferred over action $d$. It can be proven with similar arguments that this MDP admits no overtaking optimal strategy.

\section{Concluding remarks}

It remains an open problem if generic MDPs with a reachability objective admit a pure stationary 
strategy that is strongly overtaking optimal. 
To prove Theorem \ref{gen-thm}, 
we relied on techniques from linear algebra. 
When one allows non-stationary strategies, 
the transition probabilities may change period by period, 
perhaps even depending on the past play. 
Thus, in general, the transition probabilities can no longer be described by a single transition matrix, 
which complicates the analysis.

We investigated MDPs with reachability and safety objectives. 
Our solution concept of overtaking optimality for these MDPs may pave the way to define an 
overtaking solution concept for two-player zero-sum perfect information games. 
In these games, each state is controlled by one of the players. 
We can assume that Player~1's goal is to reach a certain state as quickly as possible, 
and Player~2's goal is to reach that state as slowly as possible. 
One possible solution concept for these games would be a pair of pure stationary strategies such that, given the opponent's strategy, each player's strategy is (strongly) overtaking optimal. The right notion and its existence require further investigation.

When the decision maker uses a stationary strategy, the transition probabilities are described by a transition matrix. In several situations, it can be important to study the probability distribution of the current state, at any period $t$, on condition that the state $s^*$ has not been reached yet. This conditional distribution can be expressed with the help of the transition matrix, and it converges under some conditions to a limit, called a quasi-stationary distribution. This convergence and its speed are both subject of study in the literature, see for example Diaconis and Miclo (2015).

\section*{References}
\noindent Baier, C., and Katoen, J. P. (2008): \textit{Principles of model checking,} MIT press.\medskip

\noindent Blackwell D (1962): Discrete dynamic programming. \emph{Annals of Mathematical Statistics} 33: 719-726.\medskip

\noindent Brihaye T, Bruy\`{e}re V, De Pril J (2014): On equilibria in quantitative games with reachability/safety objectives. \emph{Theory of Computing Systems} 54, 150-189.\medskip

\noindent Bruy\`{e}re, V. (2017). Computer aided synthesis: A game-theoretic approach. In Charlier, E., Leroy, J., and Rigo, M. (eds), \textit{Developments in Language Theory}, Lecture Notes in Computer Science, 10396, Springer, pp. 3--35.\medskip

\noindent Carlson DA, Haurie A, and  Leizarowitz A (1991): \textit{Infinite Horizon Optimal Control}. Berlin, Springer-Verlag.\medskip

\noindent Chatterjee K, Henzinger TA. (2012): A survey of stochastic $\omega$-regular games. \emph{Journal of Computer and System Sciences} 78: 394-413.\medskip

\noindent De Santis, E. and Spizzichino, F. (2012): 
Stochastic comparisons between hitting times for Markov Chains and words' occurrences.
Arxiv 1210.1116.
\medskip

\noindent Diaconis, P. and Laurent, M. (2015): On quantitative convergence to quasi-stationarity. \emph{Annales de la Facult\'{e} des Sciences de Toulouse: Math\'{e}matiques} 24, 973-1016.\medskip

\noindent Flesch J, Predtetchinski A, Solan E (2017): Sporadic overtaking optimality in Markov decision problems. \emph{Dynamic Games and Applications} 7, 212-228. \medskip

\noindent Guo X, Hern\'{a}ndez-Lerma O (2009): \textit{Continuous-Time Markov Decision Processes}. Springer-Verlag, Berlin.\medskip

\noindent Howard RA (1960): \emph{Dynamic Programming and Markov Processes}. MIT Press, Cambridge. \medskip

\noindent Leizarowitz, A (1996): Overtaking and almost-sure optimality for infinite horizon Markov decision processes. \emph{Mathematics of Operation Research} 21, 158--181.\medskip

\noindent M\'{e}der Z, Flesch J, Peeters R (2012): Optimal choice for finite and infinite horizons. \emph{Operation Research Letters} 40, 469-474.\medskip

\noindent Nowak AS, Vega-Amaya O (1999): A counterexample on overtaking optimality. \textit{Mathematical Methods of Operations Research} 49, 435--439.\medskip

\noindent Puterman ML (1994): \emph{Markov Decision Processes: Discrete Stochastic Dynamic Programming}. Wiley, New York.\medskip

\noindent Randour, M., Raskin, J. F., and Sankur, O. (2015): Variations on the stochastic shortest path problem. \textit{International Workshop on Verification, Model Checking, and Abstract Interpretation} (pp. 1-18).Springer, Berlin, Heidelberg. \medskip

\noindent Randour, M., Raskin, J. F., and Sankur, O. (2017): Percentile queries in multi-dimensional Markov decision processes. \textit{Formal methods in system design,}  50(2-3), 207-248. \medskip

\noindent Shapley LS (1953): Stochastic games. \textit{Proceeding of the National Academy of Sciences} 39: 1095-1100.\medskip

\noindent Stern, LE (1984): Criteria of optimality in the infinite-time
optimal control problem. \textit{Journal of Optimization Theory and Appliations} 44, 497--508.\medskip

\noindent Zaslavski AJ (2006): \textit{Turnpike Properties in the Calculus of Variations and Optimal Control}. Springer, New York.\medskip

\noindent Zaslavski AJ (2014): \textit{Turnpike Phenomenon and Infinite Horizon Optimal Control}. Springer Optimization and its Applications, New York.

\section{Appendix: Computing the horizon $T$ from which the optimal strategy in the generic case overtakes other stationary strategies}

As in Section \ref{sect-gen}, suppose that the target state $s^*$ is absorbing and that for all states $s,z\in S\setminus\{s^*\}$ and action $a\in A(s)$, the transition probability $p(z|s,a)$ is strictly positive.

Let $A_\sigma$ be the transition matrix corresponding to a stationary strategy $\sigma$. Take arbitrary states $s,z\in S\setminus\{s^*\}$. We will make use of inequalities (37) and (38) in De Santis and Spizzichino (2012) to bound the probability $A^t_{\sigma}(s,z)$ of moving from state $s$ to state $z$ in $t$ steps, under the strategy $\sigma$. 

Inequality (37) bounds $A^t_{\sigma}(s,z)$ from above. This bound uses the second highest eigenvalue of the transition matrix, that we refer to as $\lambda_2(A_\sigma)$ and states that: 
 there exists $c>0$ such that for all $t\in \NN$ we have $A^t_{\sigma}(s,z)\leq c\cdot t^{|S|} (\lambda_2(A))^t$. The constant $c$ is obtained from the Jordan form representation of the matrix $A_\sigma$. 
 
Inequality (38) obtains a lower bound for $A^t_{\sigma}(s,z)$: there exists $\widetilde{c}>0$ and $m\in\NN$ such that for all $t\geq m$ we have $A^t_{\sigma}(s,z)\geq \widetilde{c}\cdot (\lambda_2(A))^t$. The constant $\tilde{c}$ and the bound $m$ are also obtained from the Jordan form representation of the matrix $A_\sigma$.

Now take two stationary strategies $\sigma$ and $\sigma'$, and assume that $\lambda_2(A_\sigma) < \lambda_2(A_{\sigma'})$. Let $s\in S\setminus\{s^*\}$ be the initial state. Given the two inequalities above, we can find $T\in\NN$ such that $\sigma$ outperforms $\sigma'$ beyond $T$: for all $t> T$, we have $\PP_{\sigma}(t^*\leq t)> \PP_{\sigma'}(t^*\leq t)$. Indeed, choose $T\geq m$ which satisfies  $T^{|S|}\left(\tfrac{\lambda_2(A_{\sigma})}{\lambda_2(A_{\sigma '})}\right)^T <\tfrac{\tilde{c}}{c} $. Then, for all periods $t\geq T$ and all states $z\in S\setminus\{s^*\}$, we have $A^t_{\sigma}(s,z)<A^t_{\sigma'}(s,z)$. Hence, for all periods $t> T$, we have 
\[A^{t-1}_{\sigma}(s,s^*)\,=\,1-\sum_{z\in S\setminus\{s^*\}}A^{t-1}_{\sigma}(s,z)\,>\,1-\sum_{z\in S\setminus\{s^*\}}A^{t-1}_{\sigma'}(s,z)\,=\,A^{t-1}_{\sigma'}(s,s^*).\] Using $\PP_{s\sigma}(t^*\leq t)\,=\,A^{t-1}_{\sigma}(s,s^*)$ and $\PP_{s\sigma'}(t^*\leq t)\,=\,A^{t-1}_{\sigma'}(s,s^*)$, the claim follows. 

\end{document}